\documentclass{amsart}
\usepackage{epsfig}
\usepackage[english]{babel}
\usepackage{amssymb}
\usepackage{amsmath}
\usepackage{graphicx}

\newcommand{\R}{\mathbb{R}}
\newcommand{\E}{\mathbb{E}}
\newcommand{\es}{\mathbb{S}}
\newcommand{\1}{\mathbb{1}}
\renewcommand{\P}{\mathbb{P}}
\newcommand{\Z}{\mathbb{Z}}
\newcommand{\N}{\mathbb{N}}
\newcommand{\T}{\mathbb{T}}

\newcommand{\setdef}{\stackrel {\rm {def}}{=}}

\newtheorem{thank}{\ \ \ Acknowledgment}
\newcounter{tictac}

\def\1{\,\rlap{\mbox{\small\rm 1}}\kern.15em 1}

\def\build#1_#2^#3{\mathrel{\mathop{\kern 0pt#1}\limits_{#2}^{#3}}}
\def\tend#1#2{\build\hbox to 12mm{\rightarrowfill}_{#1\rightarrow #2}^{ }}

\def\converge#1#2#3#4{\build\hbox to
#1mm{\rightarrowfill}_{#2\rightarrow #3}^{\hbox{\scriptsize #4}}}

\theoremstyle{definition}

\newtheorem{thm}{Theorem}[section]
\newtheorem{prop}[thm]{Proposition}
\newtheorem{lemm}[thm]{Lemma}
\newtheorem{defn}[thm]{Definition}

\newtheorem{Prop}[thm]{Proposition}

\newtheorem{rem}[thm]{Remark}
\newtheorem{Cor}[thm]{Corollary}

\newcommand{\cp}{{\mathcal P}}

\newcommand{\vep}{\varepsilon}
\newcommand{\ov}{\overline}

\newcommand{\beq}{\begin{equation}}
\newcommand{\eeq}{\end{equation}}
\newcommand{\xbm}{(X,{\mathcal B},\mu)}

\begin{document}
\title[Approximate transitivity property and Lebesgue spectrum]{Approximate transitivity property and Lebesgue spectrum}
\author{E. H. El Abdalaoui}
\address{ Department of Mathematics, University
of Rouen, LMRS, UMR 60 85, Avenue de l'Universit\'e, BP.12, 76801
Saint Etienne du Rouvray - France}
\email{elhoucein.elabdalaoui@univ-rouen.fr }

\author {M. Lema\'nczyk }
\address{Faculty of Mathematics and Computer Science, Nicolaus Copernicus University, Chopina 12/18, 87-100, Toru\'n,
and Institute of Mathematics, Polish Academy of Sciences,
\'Sniadeckich 8, 00-950 Warszawa, Poland}
\email{mlem@mat.uni.torun.pl} \footnote{Research partially
supported by Polish MNiSzW grant N N201 384834, Marie Curie
"Transfer of Knowledge" EU program -- project MTKD-CT-2005-030042
(TODEQ) and MSRI (Berkeley) program ``Ergodic Theory and
Commutative Number Theory"}

\maketitle

{\renewcommand\abstractname{Abstract}
\begin{abstract}
Exploiting a spectral criterion for a system not to be AT we give
some new examples of zero entropy systems without the AT property.
Our examples include  those with finite spectral multiplicity --
in particular we show that the system arising from the
Rudin-Shapiro substitution is not AT. We also show that some
nil-rotations on a quotient of the Heisenberg group as well as
some (generalized) Gaussian systems are not AT. All known examples
of non AT-automorphisms contain a Lebesgue component in the
spectrum.
\vspace{8cm}\\

\hspace{-0.7cm}{\em AMS Subject Classifications} (2000): 37A15, 37A25, 37A30.
\end{abstract}
\thispagestyle{empty}
\newpage
\section{Introduction}

In this article we deal with the {\em approximate transitivity}
property (AT property for short) in ergodic theory. This property
has been introduced by A. Connes and G.J.\ Woods in {\cite
{Connes-woods}} in connection with some classification problems of
factors of type $III_0$ in the theory of von Neumann algebras. We
recall now the definition and some basic facts.

Let $G$ be an Abelian countable (discrete) group acting as measure
preserving transformations: $g\mapsto T_g$, on a standard
probability Borel space $(X,\mathcal{B},\mu)$. This action is
called AT (or AT(1)) if for an arbitrary family of nonnegative
functions $f_1,\ldots,f_l \in L_{+}^1(X,\mathcal{B},\mu), l \geq
2,$ and any $\varepsilon
>0$, there exist a positive integer $s$, $g_1,\ldots,g_s \in G$,
$\lambda_{j,k} \geq 0, j=1,\ldots l, k=1, \ldots, s$ and $f \in
L_{+}^1(X,\mathcal{B},\mu)$ such that
\begin{equation}\label{cw1}
 \|f_j-
\sum_{k=1}^{s}\lambda_{j,k}f\circ T_{g_k}\|_1 < \varepsilon,~~~~ 1
\leq j \leq l.
\end{equation}
In fact, in (\ref{cw1}) we can restrict ourselves to take only
$l=2$; indeed, given $\varepsilon>0$ we apply~(\ref{cw1}) to
$f_1,f_2$ and obtain $f=f_{12}$, then we apply again~(\ref{cw1})
to $f_{12}$ and $f_3$ for some $\varepsilon'$ sufficiently small
and obtain $f_{123}$ and we conclude after $l-1$ steps.

Only few general facts about AT-systems in ergodic theory are
known. The AT property forces the system to be ergodic and to have
zero entropy \cite{Da}, {\cite{Connes-woods}}, \cite{Dooley-Quas}. Moreover funny rank~1
systems enjoy the AT property {\cite
{Connes-woods}}, \cite{Sokhet},\cite{Sokhet2}.
Clearly, the class of AT-systems is closed under taking factors
and inverse limits (and roots for $\Z$-actions).

The action of $G$ on $\xbm$ induces a (continuous) unitary
representation, called {\em Koopman representation}, of $G$ in the
space $L^2(X,\mathcal{B},\mu)$ given by $U_{T_g}f = f\circ T_g, f
\in L^2(X,\mathcal{B},\mu)$ and $g \in G$. Recall that such a
representation is said to have {\em simple spectrum} if
$L^2_0\xbm=G(f)$ where $G(f)$ stands for the {\em cyclic space}
generated by~$f$, i.e.\ $G(f)=\overline{\mbox{span}}\{f\circ
T_g:\:g\in G\}$. In view of the definition of the AT property it
is natural to ask whether it already  implies simplicity of the
spectrum -- this question appeared (or is treated implicitly) in
several papers, see David \cite{Da}, Hawkins
\cite{Hawkins}, Hawkins-Robinson \cite{HawkinsR}, Golodets \cite{Golodets} and Dooley-Quas
\cite{Dooley-Quas}. This is still an
open problem also for $G=\Z$. A stronger conjecture due to Dooley
and Quas \cite{Dooley-Quas} is that AT-systems are exactly funny
rank-1 systems (this latter class is known to be a subclass of
simple spectrum actions). This conjecture is based on the fact
that a criterion for a system not be AT given in
\cite{Dooley-Quas} which we repeat in Section~\ref{crit} is also
sufficient for a system not to be of funny rank-$1$.

As in the definition of the AT property we deal with
$L^1$-functions, one can also ask about simplicity of the spectrum
in $L^1\xbm$ for the induced action of $G$ on $L^1$, that is we
ask whether  there exists a function $f \in L^1\xbm$ for which the
linear span of the functions $f\circ T_g$, $g \in G$, is dense in
$L^1$. The conjecture of Thouvenot from the 1980th states that
each ergodic automorphism has a simple $L^1$-spectrum (see also
related works on $L^p$-multiplicities by Iwanik
\cite{Iwanik1,Iwanik2} and Iwanik-Sam de Lazaro \cite{IwanikL}).
Thouvenot himself observed that for $G=\Z$ AT-automorphisms have
simple $L^1$-spectrum; indeed, all we need to show is  that given
$g,h\in L^1\xbm$ and $\vep>0$ we can find $f\in L^1\xbm$ such that
$$d(g,f)<\vep\;\;\mbox{and} \;\;d(h,\Z(f))<\vep$$
because then the open set $\{f\in L^1\xbm:\:d(h,\Z(f))<\vep\}$ is
dense, and we can use a Baire type argument. Now by using the AT
property, we can easily arrive at a situation that
$d(g,P(U_T)(f'))<\vep'$ and $d(h,Q(U_T)(f'))<\vep'$ where $P,Q$
are trigonometric polynomials. By replacing (in the space $A(\T)$)
$P$ by a another trigonometric polynomial we can assume that $P$
has no zeros on the circle and we simply put $f=P(U_T)f'$ noticing
that the cyclic space generated by $f$ is the same as the one
generated by $f'$; indeed  $1/P(z)$ also belongs to the space
$A(\T)$.

In fact, it was unknown until very recently that a system with
zero entropy without AT property could exist ($G=\Z$). In
\cite{Dooley-Quas} two examples of zero entropy non AT-systems are
exhibited. For both of them the associated Koopman operator has a
Lebesgue component in the spectrum,  moreover the component has
infinite multiplicity. In connection with that two natural
questions arises. Can we find a non AT ergodic automorphism whose
Koopman operator has a finite multiplicity? Does the AT property
imply singularity of the spectrum? We will give the positive
answer to the first question;  we have been unable to answer the
second one -- recall that  even in the class of rank-1
transformations it is unknown whether the spectrum has to be
singular.

In the present paper we will prove a criterion (see
Proposition~\ref{noat} below) for a system not to be AT which is
an elaborated version of an argument implicitly contained in
\cite{Dooley-Quas}. Our criterion looks spectral and should work
in case of an automorphism with a ``good'' Lebesgue component in
the spectrum. However we require for some function of type
$\chi_{P_0}-\chi_{P_1}$ ($(P_0,P_1)$ is a partition of $X$) to
have the spectral measure absolutely continuous with a good
control of its density which makes the use of the criterion ``in
practice" a delicate task. We will go through many known
constructions of zero entropy dynamical systems having a Lebesgue
component in the spectrum and  show that they or systems ``close"
to them are not AT. It should be mentioned that for all known non
AT systems the absence of approximate transitivity turned out to
be a consequence of the criterion, including the non AT property
of positive entropy automorphisms (see
Corollary~\ref{nonATpositive} below).

We show that the automorphism given by the Rudin-Shapiro
substitution is not AT but it has a finite multiplicity (more
precisely, the associated Koopman operator has the maximal
spectral multiplicity equal to~2 \cite{lemihp},
\cite{Mathew-Nadkarni}, \cite{Queffelec}). Recently Giordano and
Handelman \cite{Giordano} introduced and studied the following
notion of AT$(n)$.

\begin{defn} Let $(X,\mathcal {B},\mu,T)$ be a dynamical system and $n$ be a positive
integer. The transformation $T$ is AT$(n)$ if for any $\varepsilon
> 0$, for any set of $n+1$ nonnegative functions $\{f_i\}_{i=1}^{n+1}
\subset  L_{+}^1(X, \mu),$ there exist $n$ nonnegative functions
$g_m, m=1,\ldots,n$,  a positive integer $N$,   positive
coefficients $(\alpha^{(m)}_{i,j})^{m=1,\ldots,n}_{i=1
\ldots,n+1;j=1 \ldots n}$ and a finite sequence of integers
$\{t^{(m)}_j\}^{m=1,\ldots,n}_{j=1,\ldots,N}$ such that
\[
\|f_i -\sum_{m=1}^{n}\sum_{j=1}^{N}\alpha^{(m)}_{i,j}g_m\circ
T^{t^{(m)}_j}\|_1 < \varepsilon
\]
for all $i \in \{1, \ldots,n+1\}$.
\end{defn}

Clearly, each AT$(n)$ system  enjoys AT$(n+1)$ property.  It is
easy to see that the transformation with rank~$n$ (see e.g.\
\cite{Fe1} for the relevant definitions and properties) is AT$(n)$
and it is well known that the system determined by a substitution
on~$k$ symbols has the rank at most~$k$. In the case of
Rudin-Shapiro substitution the corresponding automorphism has
rank~$4$ (see \cite{Me}). It follows that the automorphism given
by the Rudin-Shapiro substitution is a natural example of a system
which is AT$(4)$ but not AT$(1)$, see \cite{Giordano} for other
examples  of that type.

Moreover we show that each ergodic system has an ergodic  distal
(see \cite{Fu}) extension which is not AT.  Our method  also shows
that Helson and Parry's ``random'' construction from \cite{He-Pa}
of a $2$-point extension of an arbitrary (aperiodic) system with a
Lebesgue component is ``almost" non-AT; actually (on a set of
positive measure of parameters) its $4$th power is not AT. We also
show that some nil-translations on a quotient of the Heisenberg
group are not AT. Furthermore we show that the non AT property for
affine transformations of the torus enjoy some stability property.
Finally we will deal with the non AT property in the class of
Gaussian systems. We will give examples of zero entropy (mixing)
Gaussian systems $T$ such that $T\times T\times T\times T$ is not
AT. Extensions of Gaussian systems via cocycles are treated in the
last section.

\section{A necessary combinatorial condition to have AT
property}\label{crit}

Let $G$ be a countable Abelian (discrete) group acting measurably
on a standard probability Borel space $(X,\mathcal{B},\mu)$ by
measure-preserving transformations $g\mapsto T_g$, $g\in G$. Let
$\mathcal{P}=\{P_0,P_1\}$ be a (measurable) partition of $X$. Then
each point $x\in X$ has its $\mathcal{P}$-name $\pi(x)=(x_g)_{g\in
G}\in\{0,1\}^G$, where
\[
x_g=\left\{
     \begin{array}{ll}
       0 & \hbox{if $T_g(x) \in P_0$;} \\
       1 & \hbox{if not.}
     \end{array}
   \right.
\]

Let $\Lambda$ be a finite set in $G$. We define a {\em funny word
on the alphabet} $\{0,1\}$ {\em based on} $\Lambda$ as a finite
sequence $(W_g)_{g \in \Lambda}$, $W_g \in \{0,1\}$. For any two
funny words $W,W'$ based on the same set $\Lambda$ their {\em
Hamming distance} is given by
\[
\overline{d}_{\Lambda}(W,W')=\frac1{|\Lambda|} \mbox{card}\left\{
i \in \Lambda ~~:~~W_i \neq W'_i \right \}.
\]

The proposition below gives a necessary condition for an action to
be AT. The proof of it follows word by word the proof given by
Dooley and Quas \cite{Dooley-Quas} in the case of $\Z$-actions.

\begin{prop}{(\cite{Dooley-Quas})}\label{fanny}
Let $(X,\mathcal{B},\mu,G)$ be an AT dynamical system. Then for
any $\varepsilon >0$, there exist a finite set  $\Lambda \subset
G$ of arbitrary large cardinality and a funny word $W$  based on
$\Lambda$ such that
\[
|\Lambda|\mu\{x\in X:\:
\overline{d}(\pi(x)|_{\Lambda},W)<\varepsilon\}>1-\varepsilon.
\]
\end{prop}

\section{Criterion to be non-AT}

In this section we assume that $G=\Z$ and we put $T_1=T$. We will
give a certain criterion for a system not to be AT. It is a
spectral extension of a criterion implicitly stated in
\cite{Dooley-Quas}.

\subsection{Strongly BH probability measures on the circle}
Denote by $\varepsilon_0$ the unique zero  in $(0,0.2)$ of the
polynomial $P(t)=2(1-t)(1-2t)^2-1-t$. Let $\mu$ be a probability
measure on the circle group. Motivated by Theorem~\ref{blumhanson}
below, we call $\mu$ a {\em strongly Blum-Hanson measure} (SBH
measure) if the following holds
\[
\limsup_{k \to
+\infty}\sup\left\{\left\|\frac1{\sqrt{k}}\sum_{j=1}^{k}(-1)^{\eta_j}z^{n_j}
\right\|^2_{L^2(\mu)}\!\!\!\!\!\!\!\!:n_1<\ldots<n_k,\eta_j\in\{0,1\},1\leq
j\leq k \right\} \leq 1+\varepsilon_0.
\]

\noindent{}Clearly Lebesgue measure is an SBH measure. However
more generally each absolutely continuous measure with
sufficiently ``flat'' density $g$, i.e.\ the density satisfying
$\sup_{z \in \T}g(z) < 1+\vep_0$ is also an SBH measure.

Recall that a measure $\mu$ on the circle is called a {\em
Rajchman measure} if $\displaystyle \lim_{n \to \infty}
\widehat{\mu}(n)=0$. Using some ideas of Lyons \cite{lyons} we
will prove that each SBH measure is a Rajchman measure. It is
well-known that Rajchman measure can be singular (see also
Section~\ref{gaussian} of the paper), however we have been unable
to decide whether there exist singular SBH measures. We state also in the following
the Blum-Hanson's theorem \cite{Blum-Hanson} in ergodic theory as a characterization of Rajchman
measures.

\begin{thm}{\label {blumhanson}}
 $\mu$ is Rajchman measure if and only if for any infinite increasing sequence
 $\{n_k\}_{k \in \N}$ of integers, we have
\begin{eqnarray}\label{hansonlimit}
\left\|\frac1{k}\sum_{j=1}^{k}z^{n_j}\right\|^2_{L^2(\mu)}\tend{k}{\infty}0
\end{eqnarray}
\end{thm}

Using some ideas of Lyons \cite{lyons} we  will now show that SBH
measures belong to the class of measures that annihilate all so
called $W^\ast$-sets (this class is known to be a proper subclass
of Rajchman measures, \cite{lyons}). We need to recall some basic
definitions. A sequence $\{t_n\} \subset \T$ is said to be {\em
uniformly distributed} if for all arcs $I \subset \T$
\[
\lim_{N \longrightarrow +\infty}\frac1{N}
\sum_{k=1}^{N}\chi_I(n_kx)=|I|.
\]
Weyl's criterion (e.g.\ \cite{Kuipers},pp.\ 1-3,7-8)) states that
a sequence $\{t_k \}_{k \in \N} \subset \T$  is uniformly
distributed if and only if for every non-zero integer $m$,
\[
\lim_{N \longrightarrow +\infty} \frac1{N}\sum_{k=1}^{N}e(mt_k)=0,
\]
where we use notation $e(x)$ for $e^{2i\pi x}$.

Now, following Kahane and Lyons we define the $W^*$-sets. A Borel
set $B \subset \T$ is called a $W^*$-{\em set} \cite{Kahane} (or a
non-normal set) if there exists an increasing sequence
${\{n_k\}}_{k \in \N}$ such that for every $x \in B$, $\{n_k
x\}_{k=1}^{\infty}$ is not uniformly distributed. The maximal
$W^*$-set corresponding to ${\{n_k\}}_{k \in \N}$ is the set
$W^*(\{n_k\}):= \{x\in\T:\: \{n_k x \}\;\mbox{is not uniformly
distributed}\}$ (it is Borel).

By Weyl's Criterion, in order to show that some probability
measure $\mu$ vanishes on all $W^*$-sets, we need to show that for
each $m \neq 0$ and each increasing sequence $\{n_k\}$
\begin{eqnarray}{\label {muWely}}
\frac1{K} \sum_{k=1}^{K} e(m n_k  t)=0\;\;\mbox{for
$\mu$-a.e.}\;t\in\T.
\end{eqnarray}
Actually, since $\{n_k \}$ is arbitrary, it is enough to establish
$(\ref{muWely})$ for $m=1$. We need the following form of the
strong law of large numbers for weakly correlated bounded random
variables due to Lyons \cite{lyons} and for the convenience of the reader we include the
proof.

\begin{lemm}{\label {SLLN}}
Let $\{X_n\}$ be a sequence of random variables on a probability
space $(\Omega,\mathcal{B},\P)$. Suppose that $|X_n| \leq  1$ a.s.
and there exists a positive constant $c \geq 1$ such that for any
$\{n_k\} \uparrow +\infty$ and $N$ large enough we have
\begin{eqnarray}{\label {l2majoration}}
\left\|\frac{1}{K}\sum_{k=1}^{K}X_{n_k }\right\|^2< \frac{c}{K}.
\end{eqnarray}
Then the strong law of large numbers holds:
\[
\lim_{N \longrightarrow +\infty}\frac1{K}\sum_{k=1}^{K}
X_{n_k}=0~~~~{\rm {a.~s.}}\]
\end{lemm}
\begin{proof} Assume that $n_k\uparrow +\infty$. It follows from~(\ref{l2majoration})  that
\[
\int \sum_{K \geq 1}\left |
\frac1{K^2}\sum_{k=1}^{K^2}X_{n_k}\right|^2 d\mu < c\sum_{K \geq
1}\frac1{K^2}=\frac{c \pi^2}{6}.
\]
Therefore $\P$-a.s.
\[
\sum_{K \geq 1}\left |\frac1{K^2}\sum_{k=1}^{K^2}X_{n_k}\right
|^2<\infty \] and hence ($\P$-a.s.)
\[
\lim_{K \longrightarrow +\infty}\left
|\frac1{K^2}\sum_{k=1}^{K^2}X_{n_k}\right |=0.
\]
\noindent Now if $m^2\leq K < {(m+1)}^2$, then
\[
\left |\frac{1}{K}\sum_{k=m^2+1}^{K}X_{n_k} \right|\leq
\frac1K(K-m^2)\leq\frac1K(2m+1)\leq\frac1{m^2}(2m+1)\to0
\]
when $m\to\infty$. But
\[
\left |\frac1{K^2}\sum_{k=1}^{K}X_{n_k}\right | \leq \left
|\frac1{m^2}\sum_{k=1}^{m^2}X_{n_k}\right|
+\left|\frac{1}{K}\sum_{k=m^2+1}^{K}X_{n_k} \right|\] and hence
\[
\lim_{K \longrightarrow +\infty} \frac1{K}\sum_{k=1}^{K}X_{n_k} =0
{\rm {~~a.s.}}
\]
which completes the proof.
\end{proof}

\begin{prop}{\label {wsets}}
If $\mu$ is an SBH measure then $\mu(E)=0$ for each $W^*$-set
$E\subset\T$.
\end{prop}
\begin{proof} Let $\mu$ be an SBH measure. Then, for any
increasing sequence $\{n_k \}$ we have
\[
\left\|\frac1{\sqrt{k}}\sum_{j=1}^{k}z^{n_j}\right\|^2_{L^2(\mu)}<1+\varepsilon_0
\]
for $k$ large enough. It follows from Lemma~\ref{SLLN} that the
strong law of large numbers holds for the sequence $\{X_{n_k}\}$
with  $X_{n_k}=e(n_k\cdot)$, $k\geq1$, and the result follows.
\end{proof}

We now pass to our criterion for a system not to be an AT-system
(we refer the reader to \cite{Cornfeld} to basic facts about
spectral theory of dynamical systems).

\begin{prop}\label{noat}
Assume that a dynamical system $(X,\mathcal{B},\mu,T)$ is ergodic
and  that there exists a (measurable) partition
$\mathcal{P}=\{P_0,P_1\}$ with the following properties:
\begin{enumerate}
    \item[i)] There exists $S$ in the centralizer $C(T)$ of $T$ such that
    $SP_0=P_1$; in particular  $\mu(P_0)=\mu(P_1)=\frac12$.
    \item[ii)] The spectral measure $\sigma_{\chi_{{}_{P_0}}-\chi_{{}_{P_1}}}$ of $\chi_{P_0}-\chi_{P_1}$
    is an SBH measure.
\end{enumerate}
Then the system is not AT.
\end{prop}
\begin{proof}
Let us take  $W$ a funny word based on a subset $\Lambda:
n_1<n_2<\cdots<n_k$. For $x\in X$ put
\[
\Theta^W(x)=\frac1{k}\sum_{j=1}^{k}A^W_j(x),
\]
where $A_j^W$ is defined as
\[
A_j^W(x)=\left\{
\begin{array}{rl}
    1 & \hbox{if~~ $W_{n_j}=x_{n_j}$} \\
    -1 & \hbox{if~~ not}
\end{array}%
\right.
\]
(recall that $(x_n)=\pi(x)$ is the $\mathcal P$-name of $x$). Then
the distribution $\Theta_\ast$ of $\Theta$ is symmetric. Indeed,
we have $\pi(x)=-\pi(Sx)$ and therefore
$$\Theta(\pi(x))=-\Theta(\pi(Sx))$$
and since $S$ is measure-preserving the symmetry of  $\Theta_\ast$
follows.

Notice that \begin{equation}\label{ee1}
A_j^W(x)=(-1)^{W_{n_j}}(\chi_{P_0}-\chi_{P_1})(T^{n_j}x)\end{equation}
and that \begin{equation}\label{ee2}
\Theta^W(x)=1-2\overline{d}_\Lambda(W,\pi(x)|_\Lambda).\end{equation}
In view of (\ref{ee2}), the symmetry of $\Theta_\ast$ and  the
Tchebychev inequality we obtain that
\begin{eqnarray}\label{pp}
\mu\{x\in X:~~\overline{d}_\Lambda(W,\pi(x)|_\Lambda)
<\varepsilon\}&=&\mu \{x\in
X:~~\Theta^W(x)> 1-2\varepsilon\} \nonumber \\&=&\frac12\mu \{x\in X
:~~|\Theta^W(x)|
>1-2\varepsilon\}\\ &\leq&
\frac1{2(1-2\varepsilon)^2}||\Theta^W||_2^2.\nonumber
\end{eqnarray}
But, in view of (\ref{ee1}) and the Spectral Theorem
$$
\|\Theta^W\|_2^2=\int_X\left|\frac1k\sum_{j=1}^kA_j^W\right|^2\,d\mu$$
$$
=\frac1{k^2}\sum_{i,j=1}^k\int_X(-1)^{W_{n_i}}\left(\chi_{P_0}-\chi_{P_1}\right)(T^{n_i}x)\cdot
(-1)^{W_{n_j}}\left(\chi_{P_0}-\chi_{P_1}\right)(T^{n_j}x)\,d\mu(x)$$
$$
=\frac1{k^2}
\sum_{i,j=1}^k(-1)^{W_{n_i}+W_{n_j}}\hat{\sigma}_{\chi_{{}_{P_0}}-\chi_{{}_{P_1}}}(n_i-n_j)
 =\frac1{k}\int
 \left|\frac1{\sqrt{k}}\sum_{i=1}^{k}(-1)^{W_{n_i}}z^{n_i}\right|^2\,
d{\sigma}_{\chi_{{}_{P_0}}-\chi_{{}_{P_1}}}(z).
$$
It follows that for $k$ large enough we have
\begin{equation}\label{ee3}
\|\Theta\|_2^2 < \frac1{k}(1+\varepsilon_0).
\end{equation}
Combining (\ref{pp}) and (\ref{ee3}) we obtain that
\[
k~\mu\{x\in
X:~~\overline{d}_\Lambda(W,\pi(x)|_\Lambda)<\varepsilon\} \leq
\frac{1+\varepsilon}{2{(1-2\varepsilon)}^2}
\]
and since $W$ was arbitrary this contradicts
Proposition~\ref{fanny}.
\end{proof}

\begin{Cor}\label{noatAS}
Under the assumptions of Proposition~\ref{noat} assume that the
Fourier transform of $\sigma_{\chi_{P_0}-\chi_{P_1}}$ is in $l^1$.
Then there exists $m_0\geq 1$ such that $T^m$ is not AT for all
$m\geq m_0$.
\end{Cor}
\begin{proof}
In this case $\sigma=\sigma_{\chi_{P_0}-\chi_{P_1}}$ is absolutely
continuous with the (continuous) density $d$ given by
$d(z)=\sum_{n=-\infty}^\infty \hat{\sigma}(n)z^n$ and by the same
token if instead of $T$ we consider $T^m$ the spectral measure
$\sigma_m$ of $\chi_{P_0}-\chi_{P_1}$ is also absolutely
continuous with the density $d_m$ given by
$$d_m(z)=\sum_{n=-\infty}^\infty\hat{\sigma}_m(n)z^n=\sum_{n=-\infty}^\infty\hat{\sigma}(mn)z^{n},$$
so for $m$ large enough $\sigma_m$ will be an SBH measure.
\end{proof}

\begin{rem} Notice that the $L^2$-norm of $\chi_{P_0}-\chi_{P_1}$ is~$1$. Moreover
the spectral measure of $\chi_{P_0}-\chi_{P_1}$ is Lebesgue if and
only if the sequence of partitions $\{T^n{\mathcal P}\}$ is
pairwise independent. Indeed, for $n\geq1$,
$\hat{\sigma}_{\chi_{P_0}-\chi_{P_1}}(n)=0$ implies that $$
\mu(T^{-n}P_0\cap P_0)+\mu(T^{-n}P_1\cap P_1)=\mu(T^{-n}P_0\cap
P_1)+\mu(T^{-n}P_1\cap P_0).$$ Thus $$\mu(T^{-n}P_0\cap
P_0)+\mu(T^{-n}P_1\cap P_1)=\frac12$$ and since $\mu(T^{-n}P_0\cap
P_1)+\mu(T^{-n}P_0\cap P_0)=\frac12$, we have $$ \mu(T^{-n}P_1\cap
P_1)=\mu(T^{-n}P_0\cap P_1)=\frac14$$ and therefore we obtain
pairwise independence (see \cite{Courbage-Hamdan},
\cite{Flaminio}).

However, in general it is unclear that even for transformations
with Lebesgue spectrum or the more with a Lebesgue component in
the spectrum) we can always find a partition $\cp$ satisfying the
assumptions of Proposition~\ref{noat} where in addition the
spectral measure of $\sigma_{\chi_{P_0}-\chi_{P_1}}$ is
exactly Lebesgue. A certain flexibility of Proposition~\ref{noat}
consists in the fact that for some natural partitions we need only
to show that the corresponding spectral measure is absolutely
continuous with the density sufficiently flat.

In the rest of the paper we will show how this can be applied in
practice.
\end{rem}

\begin{rem}
We can also define the notion of SBH measure for groups $G$ more
general than $\Z$. Based on Propositon~\ref{fanny} we can then
prove a relevant version of Proposition~\ref{noat} as a criterion
for a $G$-system to be non AT. It would be interesting to know
which results of Section~\ref{applications} have their natural
generalizations.
\end{rem}

\section{Applications}\label{applications}

Except for the Gaussian case considered in Section~\ref{gaussian}
all examples below of non AT-automorphisms will be given as group
extensions. Recall briefly some basic facts. Let
$T:(X,\mathcal{B},\mu) \longrightarrow (X,\mathcal{B},\mu) $ be an
ergodic automorphism. Let $G$ be a compact metric Abelian group
with Haar measure $m$. Denote by
 $\widehat{G}$ the character group of $G$. By a {\it{cocycle}} we mean  a measurable
function $ \varphi :X \to G$; in fact such a $\varphi$ generates a
cocycle $\varphi(n,\cdot)=\varphi^{(n)}(\cdot)$ by $$
\varphi^{(n)}(x)=\left\{ \begin{array}{cll}
\varphi(x)+\ldots+\varphi(T^{n-1}x) & \mbox{if} & n>0,\\
0 & \mbox{if} & n=0,\\
-(\varphi(T^nx)+\ldots+\varphi(T^{-1}x)) & \mbox{if} &
n<0.\end{array}\right.$$ Then we define a $\mu \otimes
m$-preserving automorphism
$$
T_{\varphi}:X \times G \to X \times G,\;
T_{\varphi}(x,g)=(Tx,\varphi(x)+g),~~~x \in X, g \in G
$$
called a $G$-{\em extension of} $T$. Notice that
$T_\varphi^n(x,g)=(T^nx,\varphi^{(n)}(x)+g)$. The space $L^2(X
\times G , \mu \otimes m )$ can be decomposed as
\begin{equation}\label{dec1}
L^2(X \times G , \mu \otimes m )=\bigoplus_{\chi \in
\widehat{G}}L_{\chi}, \end{equation} where each subspace $L_{\chi}
= \{f \bigotimes \chi : f \in L^2(X, \mu)\}$ is
$U_{T_{\varphi}}$-invariant, and the restriction of
$U_{T_{\varphi}}$ to $L_{\chi} $ is unitarily equivalent to
$V_{\varphi,T,\chi}~~:~~ L^2(X, \mu)\longrightarrow L^2( X , \mu
)$ defined by $V_{\varphi,T,\chi}(f)(x) = \chi(\varphi(x)) f (Tx)
$, $x \in X$. It follows that to describe  spectral properties of
$T_{\varphi}$ it is sufficient to study spectral properties of
$V_{\varphi,T,\chi}~, \chi \in \widehat{G}.$

\subsection{Two point extensions and the AT property, Rudin-Shapiro substitution}
Assume that $G=\Z/2\Z=\{0,1\}$. In this case we have a natural
partition $\cp$ given by $P_0=X \times \{0\}, P_1=X \times \{1\}$.
The assumption~(i) of Proposition~\ref{noat} is satisfied as
$SP_0=P_1$ where $S(x,i)=(x,i+1)$ is in the centralizer of
$T_{\varphi}$. Notice that the decomposition~(\ref{dec1}) of
$L^2(X \times \{0,1\})$ is of the form $L_0 \bigoplus L_1$ where
$L_0=\{f \in L^2(X \times \{0,1\} )~~:~~f\circ S=f \}$ and $L_1=\{
f \in L^2(X \times \{0,1\} ~~:~~f\circ S=-f\}.$ Thus
$U_{T_{\varphi}}|_{L_1}$ is spectrally isomorphic to the operator
$V$ defined by
\[
(V(f))(x)={(-1)}^{\varphi(x)} f(Tx).
\]
The function $\chi_{P_0}-\chi_{P_1}$ belongs to $L_1$ and under
the above isomorphism it corresponds to the constant
function~$1\in L^2\xbm$. It follows that
$$
\widehat{\sigma}_{\chi_{P_0}-\chi_{P_1}}(n)=\langle V^n1,1\rangle
$$$$= \mu\left(\{x\in X :\:\sum_{j=0}^{n-1}\varphi(T^jx)=0\}\right)-\mu\left(\{x\in X
:\:\sum_{j=0}^{n-1}\varphi(T^jx)=1\}\right).
$$

In case when $T$ is the dyadic odometer such extensions were
intensively studied in the 1980th. In particular,  Mathew and
Nadkarni \cite{Mathew-Nadkarni} gave constructions of $\varphi$
such that
\begin{equation}\label{mn}
\langle V^n1,1\rangle=0\;\;\mbox{for all}\;\;n\neq0,
\end{equation}
that is, the corresponding spectral measure is equal to Lebesgue
measure (in fact the Lebesgue component has multiplicity 2). Other
examples with  Lebesgue component with arbitrary even multiplicity
are given by so called Toeplitz extensions in \cite{lemihp} --
each time~(\ref{mn}) holds. In particular it is shown in
\cite{lemihp} that the system given by Rudin-Shapiro substitution
(see \cite{Queffelec} where also it is shown that it has a
Lebesgue component of multiplicity 2) is a particular member of
Mathew-Nadkarni's family.

\begin{Cor}
Mathew-Nadkarni's maps as well as  all examples from \cite{lemihp}
having even Lebesgue multiplicity are not AT-systems. In
particular the automorphism given by the Rudin-Shapiro
substitution is not AT.
\end{Cor}

\begin{rem}
 One may also use Ageev's construction \cite{Ageev} to produce a
continuum of weakly mixing automorphisms with spectral
multiplicity equal to~$2$ and without AT property.
\end{rem}

We now recall Helson and Parry's  construction from \cite{He-Pa}
of ``random" $2$-point extensions. Given an aperiodic automorphism
$T$ of a standard probability Borel space they give a random
construction of cocycles $\varphi_\omega:X\to\{0,1\}$ (the
parameter $\omega$ runs over a probability space $(\Omega,P)$)
such that for a.e. $\omega$, $T_{\varphi_\omega}$ has absolutely
continuous spectrum and on a set of positive $P$-measure
$T_{\varphi_\omega}$ has Lebesgue spectrum. In particular, they
prove that
\begin{equation}\label{helpar}
\int_{\Omega}\left|\int_Xe^{\pi
i\varphi^{(n)}_\omega(x)}\,d\mu(x)\right|^2dP(\omega)<\gamma_n\;\;\mbox{for
all}\;\;n\geq 4,\end{equation} where $\{\gamma_n\}$ is an
arbitrary set of positive numbers. Since, by~(\ref{helpar}), on a
set $\Omega_n\subset\Omega$ of measure at least
$1-\frac1{2^{n+1}}$ we have
$$
\left|\int_Xe^{\pi
i\varphi^{(n)}_\omega(x)}\,d\mu(x)\right|^2<2^{n+1}\gamma_n,\;\;\;n\geq4,$$
by selecting $\{\gamma_n\}$ as small as we need we can obtain that
the Fourier transform
$\{\hat{\sigma}_{\chi_{P_0}-\chi_{P_1}}(n)\}$ is absolutely
summable with
$\sum_{|n|\geq4}|\hat{\sigma}_{\chi_{P_0}-\chi_{P_1}}(n)|$ as
small as we need on the set $\cap_{n\geq4}\Omega_n$ of positive
measure of parameters. By Corollary~\ref{noatAS} (or rather its
proof) we obtain the following.

\begin{Cor}
For each ergodic (aperiodic) automorphism $T$ there exists an
ergodic $2$-point extension $T_\varphi$ such that $T^4_\varphi$ is
not AT.\end{Cor}

\noindent Therefore we get a partial answer to the question from
\cite{Dooley-Quas} (see also next section).

\subsection{Distal (ergodic) extension without AT property}

The aim of this section is to prove the following.
\begin{prop}\label{distal}
For each ergodic transformation $T$ acting on a standard Borel
probability  space $(X,\mathcal{B},\mu)$ there exists a $2$-step
group extension $\ov{T}$ which is ergodic and such that $\ov{T}$
is not AT.
\end{prop}
\begin{proof}
We use the idea of affine extension of Glasner \cite{glasner}
combined with Proposition~\ref{noat}. Let $\varphi ~:~X \to \T$,
$\T=[0,1)$, be a cocycle such that $T_{\varphi}$ is ergodic and
the groups of eigenvalues for $T$ and $T_{\varphi}$ are the same;
in particular if $T$ is weakly mixing, so is $T_{\varphi}$. Let
$\psi ~:~X \times \T \to\T$ be the cocycle defined by
$(x,y)\longmapsto y,$ and finally put
$\ov{T}=(T_{\varphi})_{\psi}$. It is easy to check (by considering
the relevant functional equations and applying Fourier series
arguments) that $\ov{T}$ is ergodic and it does not change the
group of eigenvalues of $T$.

 Consider $P_0=X \times \T \times
[0,\frac12)$. Then $P_1 \setdef P_0^c =R_{\frac12}P_0$, where
$R_{\frac12}(x,y,z)=(x,y,z+\frac12)$ and $R_{\frac12} \in
C(\ov{T})$. We have to compute the spectral measure of
$f(x,y,z)=f(z)=\chi_{P_0}(x,y,z)-\chi_{P_1}(x,y,z)=e^{i\pi\chi_A(z)},$
where $A=[0,\frac12)$. We have
\begin{eqnarray*}
&&\langle U_{\ov{T}}^nf,f\rangle=\int_{X \times \T \times \T}f(T_{\varphi}^{(n)}(x,y),\psi^{(n)}(x,y)+z)\ov{f(x,y,z)}d\mu(x)dydz\\
&=&\int e^{i\pi\chi_A{\left(\psi(x,y)+\psi(Tx,\varphi(x)+y)+\cdots+\psi(T^{n-1}x,\varphi^{(n-1)}(x)+y)+z\right)}} e^{-i\pi\chi_A(z)}d\mu(x)dydz\\
&=&\int_{X \times \T} e^{-i\pi\chi_A(z)} \left(\int_{\T}
e^{i\pi\chi_A{\left(y+\varphi(x)+y+\cdots+\varphi^{(n-1)}(x)+y+z\right)}}dy \right) d\mu(x)dz\\
&=&\int_{X \times \T} e^{-i\pi\chi_A(z)} \left(\int_{\T}
e^{i\pi\chi_A{\left(ny\right)}}dy \right) d\mu(x)dz.
\end{eqnarray*}
We will now show that
\[
\int_{\T} e^{i\pi\chi_A{\left(ny\right)}}dy=0 ~~~~~~{\rm
{for}}~~~~n\neq 0.
\]
Indeed
\[
\int_{\T}
e^{i\pi\chi_A{\left(ny\right)}}dy=\sum_{j=0}^{2n-1}\int_{I_j}
e^{i\pi\chi_A{\left(ny\right)}}dy\;\; {\rm
{where}}~~~~I_j=\left[\frac{j}{2n},\frac{j+1}{2n}\right),
\]
\noindent{}and we consider the map $y
{\stackrel{\tau_n}{\longmapsto}} ny$ and the images of $I_j$ under
this map. We have $\tau_n(I_0)=A,\tau_n(I_1)=A^c,\tau_n(I_2)=A,
\cdots$. Hence
\[
\int_{I_j}e^{i\pi\chi_A{\left(ny\right)}}dy=(-1)^{j+1}.
\]
Since $j\in\{0,\cdots,2n-1\}$, it follows that
\[
\int_{\T} e^{i\pi\chi_A{\left(ny\right)}}dy=0.
\]
We deduce that the spectral measure of $\chi_{P_0}-\chi_{P_1}$ is
exactly Lebesgue measure.
\end{proof}
\begin{rem}Instead of $\psi(x,y)=y$ we can take $\psi(x,y)=my, m\neq 0$.
In the concluding argument we divide $[0,1)$ into intervals of
length $\frac1{2|m|n}$.
\end{rem}

\begin{Cor}(\cite{Connes-woods},\cite{Da}).\label{nonATpositive}
If the dynamical system $(X,\mathcal{B},\mu,T)$ is AT then its entropy is zero.
\end{Cor}
\begin{proof}
First, note that no Bernoulli dynamical system is AT. Indeed, apply the proposition 4.4 to get the
weakly mixing 2-step compact group extension ${(T_{\phi})}_{\psi}$ of $T$ which is not AT. But, by
\cite{Rudolph2} ${(T_{\phi})}_{\psi}$ is again Bernoulli with same entropy as $T$. It follows from the Ornstein isomorphism theorem
that $T$ is not AT.\\
Assume that the entropy of $T$ is strictly positive. By the Kolmogorov-Sinai theorem there exists a Bernoulli factor with the same entropy. But the factor of the system with the AT property is AT. We get that $T$ is not AT.

\end{proof}
\subsection{Absolutely continuous cocycles over irrational
rotations without AT property}\label{ACpert} Denote $\T=[0,1)$ and
let $Tx=x+\alpha$ be an irrational rotation. Consider
$$
F(x,y)=\chi_{\T\times[0,\frac12)}(x,y)-\chi_{\T\times[\frac12,1)}(x,y)=
2\chi_{\T\times[0,\frac12)}(x,y)-1=f(y),$$ where
$f(y)=2\chi_{[0,\frac12)}(y)-1$. For $m\in\Z\setminus\{0\}$ we
have
$$
\hat{f}(m)=\int_0^1f(y)e^{-2\pi
imy}\,dy$$$$=2\int_0^1\chi_{[0,\frac12)}(y)e^{-2\pi
imy}\,dy-\int_0^1e^{-2\pi imy}=2\int_0^{1/2}e^{-2\pi imy}\,dy$$
$$
=\frac1{-\pi im}e^{-2\pi imy}|^{1/2}_0=\frac{-1}{\pi
im}\left(e^{-\pi im}-1\right)=\left\{\begin{array}{ccc}  0 &
\mbox{if} &
m=2k\\
\frac2{\pi im} &\mbox{if} & m=2k+1.\end{array}\right.
$$
Moreover
$$
F(x,y)=\sum_{m=-\infty}^\infty\hat{f}(m)e^{2\pi
imy}=\sum_{k=-\infty}^\infty\frac2{(2k+1)\pi i}e^{2\pi
i(2k+1)y}.$$ Notice that functions
$$
e^{2\pi imy}=:\xi_ m(x,y)\in L^2(\T)\otimes e^{2\pi imy}$$ where
the latter subspace is $U_{T_\phi}$-invariant for each cocycle
$\phi:\T\to\T$. In what follows we assume that
$$
\phi(x)=e^{2\pi i(x+g(x))}$$ where $g:\T\to\R$ is a ``smooth"
function: we will precise conditions imposed on $g$ later; it is
however at least that $g$ is absolutely continuous and its
derivative is a.e.\ equal to a function of bounded variation. It
follows that the spectral measure of $F$ is equal to the sum of
spectral measures of its orthogonal projections on subspaces
$L^2(\T)\otimes e^{2\pi imy}$. Let us compute the spectral measure
of $\xi_m$ :
$$
\int\xi_m\circ T_\phi^n\cdot\ov{\xi_m}\,dxdy=\int^1_0e^{2\pi
im(nx+\frac{n(n-1)}2\alpha+g^{(n)}(x))}\,dx.$$ For $n\neq0$ we
have (using integration by parts for Fourier-Stjeltjes integrals
as in \cite{Iw-Le-Ru}):
$$
\int^1_0e^{2\pi im(nx+g^{(n)}(x))}\,dx=\frac1{2\pi
im}\int_0^1\frac1{n+{g'}^{(n)}(x)}de^{2\pi im(nx+g^{(n)}(x))}$$
$$
=\frac1{2\pi im}\cdot(-1)\int_0^1e^{2\pi
im(nx+g^{(n)}(x))}d\left(\frac1{n+{g'}^{(n)}(x)}\right).$$ We will
now assume additionally that $g'>-1+\delta_0$ ($0<\delta_0<1$) so
that we can pass to a well-known estimation (see again
\cite{Iw-Le-Ru}):
$$\left|\int\xi_m\circ
T_\phi^n\cdot\ov{\xi_m}\,dxdy\right|\leq\frac1{2\pi|m|}\mbox{Var}\left(\frac1{n+{g'}^{(n)}(x)}\right)
$$
$$
\leq\frac1{2\pi|m|}\frac{\mbox{Var}{g'}^{(n)}}{(1-\delta_0)^2n^2}\leq\frac1{2\pi|m|}\frac
{\mbox{Var}\,g'}{(1-\delta_0)^2|n|}.$$ Hence (for $n\neq0$)
$$\left|\widehat{\sigma}_F(n)\right|=\left|\int F\circ T_\phi^n\cdot\ov{F}\,dxdy\right|=
\left|\sum_{m\neq0}\hat{f}(m)\int\xi_m\circ
T_\phi^n\cdot\ov{\xi_m}\,dxdy\right|
$$$$
\leq\sum_{m\neq0}|\hat{f}(m)|\frac1{2\pi|m|}\frac{\mbox{Var}(g')}{(1-\delta_0)^2|n|}=
\frac{\mbox{Var}(g')}{2\pi(1-\delta_0)^2|n|}\sum_{m\neq0}\frac{|\hat{f}(m)|}{|m|}=
\mbox{O}\left(\frac1{|n|}\right).$$ Notice however that the
constant (in the expression O$\left(\frac1{|n|}\right)$) which
appears  can be made as small as we need by assuming that
Var$(g')$ is small. It follows that $\sigma_F$ is an absolutely
continuous measure whose density $d$ is given by
$d(z)=\sum_{n=-\infty}^\infty a_nz^n$, where $a_0=1$ and for
$n\neq0$, $a_n=\widehat{\sigma}_F(n)$; moreover $|a_n|\leq
C\frac1{|n|}$ with $C>0$ as small as we need.

\begin{Prop} Assume that $Tx=x+\alpha$, $Sx'=x'+\beta$ where
$\alpha,\beta,1$ rationally independent. Let $\phi(x)=e^{2\pi i(
x+g(x))}$, $\psi(x')=e^{2\pi i(x'+h(x))}$ where $g,h:\T\to\R$ are
absolutely continuous, with the derivatives (a.e.) of bounded
variation and bounded away from $-1+\delta_0$ for some
$0<\delta_0<1$. If the variations of $g'$ and $h'$ are
sufficiently small then $T_\phi\times S_\psi$ is not AT.\end{Prop}
\begin{proof} It is enough to show that the factor $(T\times
S)_{e^{2\pi i(x+x'+g(x)+h(x'))}}$ is not AT. We consider
$$F(x,x',y)=\chi_{\T\times\T\times[0,\frac12)}-
\chi_{\T\times\T\times[\frac12,1)}.$$ By repeating all above
calculations we end up with
$$
\left|\widehat{\sigma}_F(n)\right|\leq\mbox{O}\left(\frac1{|n|^2}\right),$$
where the bounding constant is as small as we need. It follows
that the density $d(z)=1+\sum_{n\neq0}\widehat{\sigma}_F(n)z^n$ of
$\sigma_F$ is as close to~$1$ as we need and therefore $\sigma_F$
is an SBH measure.
\end{proof}

\subsection{Nil-rotations without AT property}
We consider nil-rotations $S$ in dimension~$3$ only. Hence $S$ is
defined  on $(\R^3,\ast)/_\ast\Z^3$ where we recall that the
multiplication $(x,y,z)\ast(x',y',z')=(x+x',y+y',z+z'+xy')$ on
$\R^3$ is the same as the multiplication in the Heisenberg group
of upper triangle matrices. Moreover
$$
S((x,y,z)\ast\Z^3)=(\alpha,\beta,0)\ast(x,y,z)\ast\Z^3.$$ It is
well-known (see e.g.\ \cite{Fraczek-lem}) that each such
nil-rotation is isomorphic to a skew product transformation
$T_\phi$ on $\T^2\times\es^1$ where
 $T(x,y)=(x+\alpha,y+\beta)$, $\phi=e^{2\pi i\varphi}$ and
$$
\varphi(x,y)=\alpha\{y\}-(\{x\}+\alpha)[\{y\}+\beta]+\gamma$$ (the
nil-rotation is hence ergodic if $\alpha,\beta,1$ are rationally
independent; it follows that if a nil-rotation is ergodic so are
all its non-zero powers).  It is classical (Parry) that
nil-rotations have countable Lebesgue spectrum in the
orthocomplement of the subspace of eigenfunctions.

\begin{prop}\label{nilnoat}
For every ergodic nil-rotation $S$ on $(\R^3,\ast)/_\ast\Z^3$
there exists $q\geq1$ such that $S^q$ is not AT. Moreover if in
the above representation of $S$ as a skew product $T_\phi$,
$\frac12<\beta<1$ is sufficiently close to~$1$ then $S$ is not AT.
\end{prop}
\begin{proof}The method for spectral calculations which we apply below
comes from \cite{Fraczek-lem}.

As before we consider the function
$$
F(x,y,z)=\chi_{\T\times\T\times[0,\frac12)}(x,y,z)-\chi_{\T\times\T\times[\frac12,1)}(x,y,z)=
f(z),$$ where $f(z)=2\chi_{[0,\frac12)}(z)-1$, we look at its
Fourier decomposition for $L^2(X)\otimes e^{2\pi imz}$ where
$m\in\Z\setminus\{0\}$ and we have
$$
\hat{f}(m)=\left\{\begin{array}{ccc} 0 & \mbox{if} &
m=2k\\
\frac2{\pi im} &\mbox{if} & m=2k+1.\end{array}\right.$$ It follows
that
$$\langle F\circ T_\phi^n,F\rangle=\sum_{m\neq0}\hat{f}(m)\int\xi_m\circ
T_\phi^n\cdot\ov{\xi}_m \,dxdydz$$$$=\sum_{m\neq0}\hat{f}(m)\int
e^{2\pi im\varphi^{(n)}(x,y)}dxdy.$$ As $F$ is real valued we only
need to consider $n>0$. We have
$$\int e^{2\pi im\varphi^{(n)}(x,y)}\,dxdy=\int c_n(y)\left(\int
e^{2\pi im\left(
x\cdot\sum_{j=0}^{n-1}[\{y+j\beta\}+\beta]\right)}\,dx\right)\,dy.
$$
Notice that if $q\geq1$ is so that
\begin{equation}\label{nil1}q\beta>1\;\;\mbox{then}
\sum_{j=0}^{q-1}[\{y+j\beta\}+\beta]\geq1\;\; \mbox{for each}\;
y\in[0,1).\end{equation} It follows that for each $n\neq0$
$$
\langle F\circ T_\phi^{nq},F\rangle=0$$ and the spectral measure
$F$ for $T_\phi^q$ is purely Lebesgue (cf.\
Corollary~\ref{noatAS}), which completes the first part of the
proposition.

If $\frac12<\beta<1$ in view of~(\ref{nil1}) for $n=2$ (and for
all $n\geq 2$) we have $\langle F\circ T_\phi^n,F\rangle=0$. Hence
we have only to control the case $n=1$. We have
$$
\langle F\circ T_\phi,F\rangle=\sum_{m\neq0}\hat{f}(m)\int e^{2\pi
im\varphi(x,y)}\,dxdy$$$$=\sum_{m\neq0}\hat{f}(m)\int e^{2\pi im
(\alpha\{y\}-(\{x\}+\alpha)[\{y\}+\beta])}dxdy$$
$$
=\sum_{m\neq0}\hat{f}(m)\left(\int_0^{1-\beta} b_m(y)\left(\int
e^{2\pi im\left( x\cdot[\{y\}+\beta]\right)}\,dx\right)\,dy\right.
$$$$\left.+ \int_{1-\beta}^1 b_m(y)\left(\int e^{2\pi im\left(
x\cdot[\{y\}+\beta]\right)}\,dx\right)\,dy)\right)$$
$$
=\sum_{m\neq0}\hat{f}(m)\int_0^{1-\beta} b_m(y)\left(\int e^{2\pi
im\left( x\cdot[\{y\}+\beta]\right)}\,dx\right)\,dy=
\sum_{m\neq0}\hat{f}(m)\int_0^{1-\beta} b_m(y)\,dy$$
$$
=\sum_{m\neq0}\hat{f}(m)e^{2\pi im\alpha}\int_0^{1-\beta}e^{2\pi
im\alpha\{y\}}\,dy=\sum_{m\neq0}\hat{f}(m)e^{2\pi
im\alpha}\frac1{2\pi im\alpha}\left(e^{2\pi
im\alpha\{1-\beta\}}-1\right).$$ Now, the series
$\sum_{m\neq0}\hat{f}(m)e^{2\pi im\alpha}\frac1{2\pi im\alpha}$ is
absolutely summable and the functions $\beta\mapsto e^{2\pi
im\alpha\{1-\beta\}}-1$ are continuous. It follows that if
$1/2<\beta<1$ is sufficiently close to~$1$ then the density of
$\sigma_F$ (which is a trigonometric polynomial of degree~$1$) is
as close to~$1$ as we want, and therefore $\sigma_F$ is SBH, so
the corresponding nil-rotation is not AT.
\end{proof}

\subsection{Gaussian systems and the AT property}\label{gaussian}

In this section we will study the non AT property in the class of
zero entropy Gaussian dynamical systems. We will show the
existence of mixing zero entropy Gaussian systems for which the
4th Cartesian product is not AT. Let us recall the definition of
Gaussian systems.

\begin{defn}
Given a symmetric Borel probability measure $\sigma$ on the circle
$[0,1)$ we call (real) Gaussian system of spectral measure
$\sigma$ the dynamical system
$(\Omega,\mathcal{A},T_{\sigma},\mu)$ where:

\begin{itemize}
  \item $\Omega$ is $\R^{\Z}.$
  \item $\mathcal {A}$ is the borelian $\sigma$-algebra.
  \item $S$ is the shift : ${(T_{\sigma}(\omega))}_{n}=\omega_{n+1}$.
  \item $\mu$ is defined on the cylinder by
  $\mu(\omega_{j_1}\in A_1,\cdots,\omega_{j_n}\in A_n)$ is the probability
  of visiting the set $A_1\times \cdots \times A_n$ for a Gaussian vector
  $(X_{j_1},\cdots,X_{j_n})$ of zero mean and covariances
      \[
      {\rm {Cov}}(X_{j_s},X_{j_t})=\int_{\T} z^{j_s-j_t} d\sigma(z).
      \]
\end{itemize}

\end{defn}
Such a system is then generated by real  stationary (centered)
Gaussian process, namely
 \[
X_n=X_0 \circ T_{\sigma}^{n},~~ n\in \Z, {\rm {~~where~~}}
X_0(\omega)=\omega_0.
\]
The basic account on the spectral analysis of Gaussian dynamical
systems may be found in \cite{Cornfeld}. We recall that the
maximal spectral multiplicity of every Gaussian dynamical systems
is $+ \infty $  or~$1$.

Let us point out that the transformation $S~:~(\omega_n)
\longmapsto (-\omega_n)$ preserves the Gaussian measure and
commute with any Gaussian system. In addition, for the partition
$\mathcal{P}=\{P_0,P_1\}$, given by
 $P_0=\{X_0>0 \}$, we have $SP_0=P_1.$\\
 With this remark in mind we will compute the Fourier coefficients
 of the spectral measure
 $\sigma_{\chi_{{}_{P_0}}-\chi_{{}_{P_1}}}$.

\begin{lemm}\label{arcsin}
Let $X =(X_n)_{n \in \Z}$ be a stationary centered Gaussian
process with spectral measure $\sigma$ satisfying $\hat{\sigma}(n)
\in (-1,1)$ for $n \neq 0$. Then
\[
\mu\{X_0 > 0, X_n > 0)=\frac14+\frac1{2\pi}
\arcsin(\widehat{\sigma}(n)).
\]
\end{lemm}
\begin{proof}
Put $Z_n=X_n-\hat{\sigma}(n)X_0$. It follows that $Z_n$ and $X_0$
are independent, the distribution of $Z_n$  is Gaussian with
variance $1-{\hat{\sigma}(n)}^2$ and for any Borel function
$\phi:\R\times\R\to\R$ we have
\[
\E(\phi(X_0,Z_n)|X_0)=h(Z_n),{~~~~{\rm {where~~}}}
h(z)=\E(\phi(X_0,z)).
\]
Now, by taking $\phi(x,z)=1$ if $x>0$ and $z>-\hat{\sigma}(n)x$,
and 0 otherwise we obtain that
$$
\mu\{X_0 > 0, X_n > 0 \}=\E\left(\E(\chi_{\{X_0 >0\}} \chi_{\{Z_n
> -{\hat{\sigma}(n)}X_0\}}|X_0)\right)$$
$$
=\E\left(\chi_{\{X_0 >0\}}(\E(\chi_{\{Z_n
> -{\hat{\sigma}(n)}X_0\}}|X_0)\right)
$$$$=\E\left(\chi_{\{X_0>0\}}
\mu\left(\frac{Z_n}{\mbox{Var}\,Z_n}
>\frac{-\hat{\sigma}(n)X_0}{\sqrt{1-\hat{\sigma}(n)^2}}\left|X_0\right.\right)\right)$$$$=
\E\left(\chi_{\{X_0
>0\}}\left\{1-G\left (
\frac{-{\hat{\sigma}(n)}X_0}{\sqrt{1-{\hat{\sigma}(n)}^2}}\right)
\right\}\right)
$$
where $\displaystyle G(x)=\frac1{\sqrt{2\pi}}\int_{-\infty}^{x}
e^{-\frac{u^2}2}du$. But, since $1-G(-x)=G(x)$, we get
\begin{eqnarray*}
\mu\{X_0 > 0, X_n > 0 \}&=& \E\left(\chi_{\{X_0 >0\}}
\left\{G \left( \frac{{(\hat{\sigma}(n))}X_0}{\sqrt{1-{\hat{\sigma}(n)}^2}}\right) \right\}\right)\\
&=&\int_{0}^{\infty}\left\{G \left(
\frac{{\hat{\sigma}(n)}u}{\sqrt{1-{\hat{\sigma}(n)}^2}}\right)
\right\} G'(u)du.
\end{eqnarray*}
An easy calculation by taking the derivatives of the following
functions  yields
\[
\int_{0}^{+\infty}G(au)G'(u)du=\frac1{2\pi}\arctan{a}+\frac14,
{\rm {~~~for~~any~~}} a \in \R,
\]
and by taking
$a=\frac{\hat{\sigma}(n)}{\sqrt{1-\hat{\sigma}(n)}^2}$ we
conclude.
\end{proof}

\noindent{}It follows from the lemma above that we have
\[
\widehat{\sigma}_{{\chi_{{}_{P_0}}-\chi_{{}_{P_1}}}}(n)=\frac2{\pi}\arcsin(\widehat{\sigma}(n)).
\]

We now pass to a similar problem but in the Cartesian product $(X
\times X,\mathcal{B},\mu \otimes \mu, T \times T)$ in which we
consider the process $(Y_n)$ where $Y_n=Y_0\circ (T\times T)^n$
and $Y_0(\omega_1,\omega_2)=X_0(\omega_1)X_0(\omega_2)$. We first
must extend Lemma~\ref{arcsin}.

 \begin{lemm}\label{arcsinsquare}
Let $X =(X_n)_{n \in \Z}$ be a stationary  centered Gaussian
process with spectral measure $\sigma$ such that $\hat{\sigma}(n)
\in (-1,1)$ for $n \neq 0$ and let
$Y_n(\omega_1,\omega_2)=X_n(\omega_1)X_n(\omega_2)$. Then
\[
\mu(\{Y_0 > 0, Y_n > 0\})=\frac14+\frac1{\pi^2}
{\arcsin^2(\widehat{\sigma}(n))}.
\]
\end{lemm}
\begin{proof} It is easy to see that
\begin{eqnarray*}
\mu \otimes \mu \{Y_0 >0, Y_n>0\}&=&{\left(\mu \{X_0>0,X_n
>0\}\right)}^2+{\left(\mu \{X_0<0,X_n
>0\}\right)}^2\\&+&{\left(\mu \{X_0>0,X_n <0\}\right)}^2+
{\left(\mu \{X_0<0,X_n <0\}\right)}^2.
\end{eqnarray*}
Using Lemma \ref{arcsin} and the fact that $S\in C(T)$ we obtain
that
$$\mu (\{X_0<0,X_n <0\})=\mu(\{S\{X_0<0,X_n<0\}\})
$$$$=\mu(\{X_0>0,X_n>0\})=\frac14+\frac1{2\pi}
\arcsin(\widehat{\sigma}(n)),$$ and then
\begin{eqnarray*}
\mu \{X_0<0,X_n >0\}&=&\mu \{X_0<0\}-\mu \{X_0<0,X_n <0\}=\frac14-\frac1{2\pi} \arcsin(\widehat{\sigma}(n)),\\
\mu \{X_0<0,X_n >0\}&=&\frac14-\frac1{2\pi}
\arcsin(\widehat{\sigma}(n)).
\end{eqnarray*}
\noindent{} and the proof of the lemma is complete.
\end{proof}

Since $X_n(\omega)=\omega_n$
($\omega=(\omega_n)_{n\in\Z}\in\R^{\Z}$) and $S\times Id$ is in
the centralizer of $T\times T$,
$$
\mu\otimes\mu(\{Y_0<0,Y_n<0\})=\mu\otimes\mu(S\times\,Id(\{Y_0<0,Y_n<0\}))$$$$=
\mu\otimes\mu(\{Y_0>0,Y_n>0\})=\frac14+\frac1{\pi^2}
{\arcsin(\widehat{\sigma}(n))}^2$$ and since
$\mu\otimes\mu(\{Y_0<0\})=\frac12$,
$$
\mu\otimes\mu(\{Y_0<0,Y_n>0\})=\mu\otimes\mu(\{Y_0>0,Y_n<0\})=\frac14-\frac1{\pi^2}
{\arcsin(\widehat{\sigma}(n))}^2.$$ Consider
$\overline{T}=(T\times T)\times(T\times T)$ with measure
$\overline{\mu}=(\mu\otimes\mu)\otimes(\mu\otimes\mu)$ and the
process $(Z_n)_{n\in\Z}$ where $Z_n=Z_0\otimes\overline{T}^n$ and
$$
Z_0(\omega^{(1)},\omega^{(2)},\omega^{(3)},\omega^{(4)})=
Y_0(\omega^{(1)},\omega^{(2)})Y_0(\omega^{(3)},\omega^{(4)})).$$
By the argument from the beginning of the proof of
Lemma~\ref{arcsinsquare} it follows that
$$
\overline{\mu}(\{Z_0>0,Z_n>0\})=\frac14+\frac4{\pi^4}\arcsin^4(\hat{\sigma}(n)).$$
Therefore if we put $P'_0=\{Z_0<0\}$ and $P_1'=\{Z_0\geq0\}$ then
for any $n \in \Z$ we have
\[
\widehat{\sigma}_{{\chi_{{}_{P'_0}}-\chi_{{}_{P'_1}}}}(n)=\frac{16}{\pi^4}{(\arcsin(\widehat{\sigma}(n)))}^4.
\]
Finally notice that $P_1'=S\times Id\times Id\times Id(P_0')$,
where $S\times Id\times Id\times Id\in C(\overline{T})$.

Before we formulate  the main result of this section we recall the
following result of K\"{o}rner \cite{Korner} (which is a
strengthening of a result of Iva\v{s}\"{e}v-Musatov
\cite{brownhewitt}).

\begin{thm}\label{korner}
Assume that $\phi:[1,\infty]\to[1,\infty)$ is a continuous
positive function such that\\
(i) $\int_1^\infty \phi(x)^2\,dx=+\infty,$\\
(ii) $\left(\exists K>1\right)\;\;\;K\phi(x)\geq \phi(y)\geq
\phi(x)/K\;\;\mbox{for}\;2x\geq y\geq x\geq1$.

Then there exists a singular probability measure $\sigma$ on $\T$
such that for each $n\neq0$,
$$\left|\widehat{\sigma}(n)\right|\leq\phi\left(|n|\right).$$
\end{thm}

\vspace{2ex}

 We easily verify that for each $c>0$ the function
$\phi(x)=c/\sqrt{x}$ satisfies the assumptions of
Theorem~\ref{korner}.

Notice that if $\sigma$ satisfies the assertion of
Theorem~\ref{korner} then so does the measure
$\tilde{\sigma}(A)=\sigma(\overline{A})$  (since
$|\hat{\sigma}(-n)|=|\hat{\sigma}(n)|$) and also \beq\label{sy}
\left|\left(\frac12(\sigma+\tilde{\sigma}\right)\hat{}(n)\right|\leq\phi(|n|).\eeq
In other words the assertion of Theorem~\ref{korner} holds in the
class of symmetric measures.

\begin{prop}\label{GnoAT}
There exists a mixing Gaussian zero entropy dynamical system
$(X,\mathcal{B},T,\mu)$ such that $T\times T\times T\times T$ is
not AT.
\end{prop}
\begin{proof} Using Theorem~\ref{korner} (and~(\ref{sy})) we can find $\sigma$ a probability
symmetric singular measure on the circle $\T$ such that for
$n\geq1$ \begin{eqnarray}\label{dominecoef} |\widehat{\sigma}(n)|
\leq  \frac{c}{\sqrt{n}}.
\end{eqnarray}
The constant $c>0$  has to be small enough so that for $|x|\leq
\frac c{\log2}$ we have $|\arcsin x|\leq 2|x|$ and we assume that
$$
c\leq \pi^{1/2}\left(\frac{1+\vep_0}{86}\right)^{1/4}.$$ We now
have
$$
\sum_{k\neq0}
|\widehat{\sigma}_{{\chi_{{}_{P'_0}}-\chi_{{}_{P'_1}}}}(k)|\leq
\frac{32}{\pi^4}\sum_{k \geq1}(\arcsin(\widehat{\sigma}(k)))^4\leq
\frac{512}{\pi^4}\sum_{k \geq1}(\widehat{\sigma}(k))^4$$$$ \leq
\frac{512c^4}{\pi^4}\sum_{k \geq1}\frac1{k^2}\leq 1+\vep_0.
$$
It follows that ${\sigma}_{{\chi_{{}_{P'_0}}-\chi_{{}_{P'_1}}}}$
is an SBH measure and the proof of the proposition is complete.
\end{proof}

{\bf Question.} Is it true that for {\em every} automorphism $T$
its Cartesian square $T\times T$ does not have the AT property?

\begin{rem} In the recent paper \cite{Tkorner}, K\"orner shows
that given $\alpha\in(\frac12,1)$ there is a singular measure
$\mu$ on the circle such that $\mu\ast\mu$ is absolutely
continuous with density of Lipshitz class $\alpha-\frac12$. However the
measure $\mu$ is not symmetric. If K\"orner's construction can be
``symmetrized" then we would obtain a zero entropy Gaussian
automorphism whose Cartesian square $T\times T$ is not AT.
\end{rem}

\subsection{Gaussian cocycles and non AT property}

Following \cite{Le-Le-Sk} in this section we consider extensions
of Gaussian systems via Gaussian cocycles. So we assume that
$(X_n)_{n\in\Z}$ is a stationary centered real Gaussian process
inducing the dynamical system $T=T_\sigma$ acting on
$(\Omega,\mu_\sigma)$, where $\Omega=\R^\Z$ and the probability
measure $\sigma$, always assumed to be continuous, is the spectral
measure of the process (and we can assume that
$X_n(\omega)=\omega_n$). We then consider the skew product
$T_{e^{2\pi iX_0}}$ acting on
$(\Omega\times\es^1,\mu_\sigma\otimes\lambda)$.

It has been proved in \cite{Le-Le-Sk}  that if
$\hat{\sigma}(n)\geq0$ for each $n\in\Z$ then $T_{\exp(2\pi
iX_0)}$ has countable Lebesgue spectrum in the orthocomplement of
$L^2(\Omega,\mu_\sigma)\otimes1$. Based on the proof of this
result we will now show the following.

\begin{prop}Assume that $\hat{\sigma}(n)\geq0$ for each $n\in\Z$. Then there exists $m_0$ such that for each $m\geq m_0$,
$T^m_{\exp(2\pi iX_0)}$ is not AT.
\end{prop}
\begin{proof} Take the partition $(A,A^c)$ of $\es^1$ into the upper and the
lower semicircle and consider $(\Omega\times A,\Omega\times A^c)$.
As in Section~\ref{ACpert} we notice that the Fourier
decomposition of $F(\omega,z)=(\chi_{\Omega\times
A}-\chi_{\Omega\times A^c})(\omega,z)$ is of the form $$
F(\omega,z)=f(z)=\sum_{k=-\infty}^\infty \frac2{(2k+1)\pi
i}z^{2k+1}.$$ Proceeding as in Section~\ref{ACpert}  we obtain
that
\begin{equation}\label{est1}
\left|\widehat{\sigma}_F(n)\right|=\left|\sum_{m\neq0}\hat{f}(m)\int_\Omega
e^{2\pi imX_0^{(n)}}\,d\mu_\sigma\right|.\end{equation} An
elementary calculation using the fact that the Fourier transform
of $\sigma$ is positive (see \cite{Le-Le-Sk}) shows that
$$
\|X_0^{(n)}\|^2_2\geq |n|\hat{\sigma}(0)=|n|$$ and therefore
$$ \left|\int_\Omega e^{2\pi imX_0^{(n)}}\,d\mu\right|=
e^{-2\pi^2m^2\|X_0^{(n)}\|_2^2}\leq e^{-Cm^2|n|}.$$ In view of
(\ref{est1}) it follows that the Fourier transform of $\sigma_F$
still decreases exponentially and the assertion follows from
Corollary~\ref{noatAS}.
\end{proof}

\begin{rem}
It follows that for each Gaussian system $T_\sigma$ where
$\sigma=\eta\ast\eta$ we have a Gaussian cocycle such that the
corresponding skew product has a power which is not AT.
\end{rem}

\begin{thank}\em
The first author would like to express thanks to J-P. Thouvenot
and B. Host for formulating to him  the question about existence
of finite multiplicity non-AT transformation. It is also a
pleasure for him to express thanks to the university of Toru\'n
and to the organizers of the Grefi-Mefi colloque ``From dynamics
to statistical mechanics, Luminy, CIRM, 2008" where a part of this
work has been done.

The second author would like to thank Emmanuel Lesigne and Anthony
Quas for fruitful discussions on the subject.
\end{thank}

\bibliographystyle{amsplain}
\setlength{\parsep}{0cm} \small

\end{document}